\documentclass[a4paper,11pt]{amsart}

\usepackage{amsfonts,latexsym,rawfonts,amsmath,amssymb,amsthm, mathrsfs, lscape}
\usepackage[all]{xy}
\usepackage[english]{babel}
\usepackage[utf8]{inputenc}
\usepackage{graphicx}
\usepackage{booktabs}
\usepackage{array, tabularx}
\usepackage{setspace}
\usepackage{textcomp}
\usepackage{tikz}

\usepackage[hypertexnames=false,
backref=page,
    pdftex,
    pdfpagemode=UseNone,
    breaklinks=true,
    extension=pdf,
    colorlinks=true,
    linkcolor=blue,
    citecolor=blue,
    urlcolor=blue,
]{hyperref}

\usepackage[top=1.5in, bottom=.9in, left=0.9in, right=0.9in]{geometry}

\newcommand{\referenza}{}

\newtheorem{thm}{Theorem}[section]
\newtheorem*{thm*}{Theorem \referenza}
\newtheorem{definition}[thm]{Definition}

\newtheorem*{cor*}{Corollary \referenza}
\newtheorem{lem}[thm]{Lemma}
\newtheorem*{lem*}{Lemma \referenza}

\newtheorem{prop}[thm]{Proposition}
\newtheorem*{prop*}{Proposition \referenza}

\newtheorem*{conj*}{Conjecture \referenza}
\newtheorem{rmk}[thm]{Remark}
\newtheorem*{rmk*}{Remark}

\newtheorem{defi}[thm]{Definition}

\numberwithin{equation}{section}

\def \N {\mathbb N}

\def \H {\mathbb H}

\def \R {\mathbb R}
\def \C {\mathbb C}
\def \Z {\mathbb Z}

\allowdisplaybreaks[1]

\title{Slice-quaternionic Hopf surfaces}

\author{Daniele Angella}
\address[Daniele Angella]{
Dipartimento di Matematica e Informatica ``Ulisse Dini''\\
Universit\`a degli Studi di Firenze\\
viale Morgagni 67/a\\
50134 Firenze, Italy
}
\email{daniele.angella@gmail.com}
\email{daniele.angella@unifi.it}

\author{Cinzia Bisi}
\address[Cinzia Bisi]{Dipartimento di Matematica e Informatica\\
Universit\`a di Ferrara\\
Via Machiavelli 35\\
44121 Ferrara, Italy
}
\email{cinzia.bisi@unife.it}

\thanks{The first author is supported by project PRIN ``Variet\`a reali e complesse: geometria, topologia e analisi armonica'', by project FIRB ``Geometria Differenziale e Teoria Geometrica delle Funzioni'', by project SIR2014 ``Analytic aspects in complex and hypercomplex geometry'' code RBSI14DYEB, and by GNSAGA of INdAM.
The second author was partially supported by Prin 2015 Protocollo: 2015A35N9B\_013, by Firb 2012 Codice: RBFR12W1AQ-001 and by GNSAGA-INdAM}

\date{\today}

\begin{document}

\begin{abstract}
 We investigate slice-quaternionic Hopf surfaces.
 In particular, we construct new structures of slice-quaternionic manifold on $\mathbb{S}^1\times\mathbb{S}^7$, we study their group of automorphisms and their deformations.
\end{abstract}

\maketitle

\section*{Introduction}

Holomorphic functions on open domains of $\mathbb C^n$ yield a category, whence the notion of complex manifold. In fact, just a directed graph structure may suffice to get models for constructing structures on manifolds. In particular, we are interested in slice-regular functions over the quaternions in the sense of \cite{ghiloni-perotti}, see Definition \ref{defsurfreg} adapted to the two variables. Such a notion has its source in the work by Fueter, further developed by Ghiloni and Perotti. It represents a counterpart in several variables of the notion of slice-regular function in one quaternionic variable studied in \cite{gentili-struppa-cras, gentili-struppa}, which appeared to share with holomorphic functions a rich theory from the analytic point of view, \cite{bisi-stoppato-ind},\cite{bisi-stoppato-land}; see also \cite{colombo-sabadini-struppa}. We refer to \cite{gentili-stoppato-struppa, ghiloni-perotti} for precise definitions and for results.

The first examples of manifolds modelled over quaternions are constructed and studied in \cite{salamon, sommese, bisi-gentili, gentili-gori-sarfatti}:
quaternionic tori \cite{bisi-gentili}, quaternionic projective spaces $\mathbb H\mathbb P^1$ and $\mathbb H \mathbb P^n$ \cite{salamon}, (affine) Hopf quaternionic manifolds \cite{sommese}, blow-up of $\mathbb H^n$ at $0$ and, more in general, of a quaternionic manifold at a point \cite{gentili-gori-sarfatti}, quaternionic Iwasawa manifolds \cite{sommese}, affine quaternionic manifolds \cite{gentili-gori-sarfatti, gentili-gori-sarfatti-2}. But for some exceptions \cite{gentili-gori-sarfatti}, they are indeed affine quaternionic structures. \\
In Physics also, large classes of examples of noncommutative finite-dimensional manifolds have been exhibited in connection with the Yang-Baxter equation: for example, the so called $\theta-$deformations of $\mathbb{C}^2$ (identified with $\mathbb{R}^4$ and $\mathbb{H}$), $\mathbb{C}^2_{\theta},$ studied in \cite{connes-duboisviolette} and in \cite{connes-landi} with their natural quantum groups of symmetries which are $\theta-$deformations of the classical groups $GL(4, \mathbb{R})$, $ SL(4,\mathbb{R})$ and $GL(2,\mathbb{C}).$

\medskip

This note has the aim to extend the above class of examples. More precisely, we construct and investigate {\em slice-quaternionic primary Hopf surfaces}, namely, manifolds whose universal cover is $\H^2\setminus\{(0,0)\}$ and the fundamental group equals the infinite cyclic group $\Z$, endowed with a structure of slice-quaternionic manifold.
In a sense, these are the simplest examples other than tori. Notice indeed, as in \cite{sommese}, that these constructions correspond to the two ways of constructing affine complex manifolds in dimension $1$, namely, $\left.\C\middle\slash(\Z\oplus\sqrt{-1}\,\Z)\right.$ and $\left.\C^*\setminus\{0\}\middle\slash\Z\right.$. (For the affine complex case, see also \cite{vitter}.)
This is done in view to further understand a possible notion of manifold in the slice-quaternionic class.

Other than quaternionic affine structures \cite{sommese}, case {\itshape (A)} in Theorem \ref{thm:hopf-quat}, we get new slice-quaternionic structures: take $\lambda\in\H\setminus\{0\}$, $p\in\N\setminus\{0,1\}$, $\beta\in\mathbb{H}$ with $0<|\beta|<1$, (in a moment, we will restrict to $\beta\in\R$) and define
$$ f(z,w) \;:=\; \left(z\cdot \beta^p+w^p\cdot \lambda,\; w\cdot \beta\right) \;. $$
As for notation, recall that any (non-real) quaternion $q\in\mathbb H$ can be written (uniquely) as $q=x+y \cdot I$ with $I\in\mathbb S^2=\{q\in\H \;:\; q^2=-1\}$ and $y\geq0$. We denote $I_q:=I$, and we set $L_{I_q}:=\R\oplus \R\cdot I_q\simeq\C$. In case $q\in\R$, we set $L_{I_q}=\mathbb{H}$.
Moreover, $\mathrm{ext}$ denotes the regular extension:
it acts on a function defined on $\mathbb{H}\times L_{I_\beta}$ giving a slice-regular extension on $\mathbb{H}^2$. Consider then
$$ \Gamma \;:=\; \left\{\mathrm{ext}(f\lfloor_{\H \times L_{I_\beta}}^{\circ k}) = f^{\circ k} \;:\; k\in\Z \right\} \;. $$
Note that, in order to $\Gamma$ being a group of automorphisms in the sense of \cite{ghiloni-perotti} in $\mathrm{Aut}(\H^2\setminus\{(0,0)\})$, we have to force $\beta\in\R$. 
In fact, in this case, this allows us to avoid regular extension.
Then we consider
$$ \left. \H^2 \setminus\{(0,0)\} \middle\slash \Gamma \right. \;, $$
see case {\itshape (B)} in Theorem \ref{thm:hopf-quat}.

We prove the following result.
\renewcommand{\referenza}{\ref{thm:hopf-quat}}
\begin{thm*}
 Let $f\colon \H^2\to \H^2$ be the function
  \begin{equation}\label{eq:f}
   f(z,w) \;:=\; \left( z \cdot \alpha + w^p \cdot \lambda, \, w \cdot \beta \right) \;,
  \end{equation}
 where $p\in\N\setminus\{0\}$, $\alpha,\beta,\lambda\in\H$ are such that
 $$ 0 \;<\; \left|\alpha\right| \;\leq\; \left|\beta\right| \;<\; 1 \qquad \text{ and } \qquad \left(\alpha-\beta^p\right)\cdot \lambda \;=\; 0 \;. $$
 In the following cases, the quotient of $\H^2\setminus\{(0,0)\}$ by the subgroup generated by $f$ yields a structure of slice-quaternionic Hopf surface:
 \begin{description}
  \item[Case A.]\ 
   \begin{description}
    \item[Case A.1.] when $\lambda=0$, $\alpha=\beta\in\H$ with $0<|\alpha|<1$;
    \item[Case A.2.] when $\lambda=0$, $\alpha,\beta\in\H$ with $0<|\alpha|\leq|\beta|<1$ and $\alpha\neq\beta$;
    \item[Case A.3.] when $\lambda\in\H$ with $\lambda\neq0$, $p=1$, $\alpha=\beta\in\H$ with $0<|\alpha|<1$;
   \end{description}
  \item[Case B.] when $\lambda\in\H$ with $\lambda\neq0$, $p\in\N$ with $p>1$, $\beta \in\R$ with $0<|\beta|<1$, and $\alpha=\beta^p$.
  \end{description}
\end{thm*}

\begin{rmk*}
We wonder whether other slice-quaternionic structures on $\mathbb{S}^1\times\mathbb{S}^7$ may be constructed; see Remark \ref{rmk:other}.
\end{rmk*}

The automorphism groups of the above slice-quaternionic structures are investigated in Theorem \ref{thm:aut-hopf}. Notice that, in general, a slice-quaternionic structure does not induce a holomorphic structure; compare also Remark \ref{rmk:affine-cplx}.
(Note indeed that $(X_0+X_1 \cdot J)\cdot (Y_0+Y_1 \cdot J) = (X_0\cdot Y_0-X_1\cdot \bar Y_1)+(X_0\cdot Y_1+X_1 \cdot \bar Y_0) \cdot J$ for $X_0,X_1,Y_0,Y_1\in\R\oplus I\cdot\R$, where $I,J$ are orthogonal complex structures on $\R^4$.)
Therefore the slice-quaternionic Hopf surfaces in case {\itshape (B)} do not underlie a complex Calabi-Eckmann structure.

We prove the following result.

\renewcommand{\referenza}{\ref{thm:aut-hopf}}
\begin{thm*}
Let $X = \left. \H^2 \setminus\{(0,0)\} \middle\slash \langle f \rangle \right.$ be a slice-quaternionic Hopf surface, where $f$ is as in equation \eqref{eq:f}. The dimension of the group of automorphisms of $X$ is as follows:
\begin{description}
 \item[Case A.1.] $\dim_\R \mathrm{Aut}(X)\in\{8,16\}$;
 \item[Case A.2.] $\dim_\R\mathrm{Aut}(X)\in\{4,6,8\}$;
 \item[Case A.3.] $\dim_\R\mathrm{Aut}(X)\in\{4,8\}$;
 \item[Case B.] $\dim_\R\mathrm{Aut}(X)=5$.
\end{description}
\end{thm*}

Finally, we provide families of slice-quaternionic Hopf surfaces, connecting cases {\itshape (A.1)} and {\itshape (A.3)}, respectively cases {\itshape (A.2.1)} and {\itshape (B)}, see Section \S\ref{sec:deformations}.

\begin{rmk}
As suggested by the anonymous Referee, we observe that analogous constructions can be performed to obtain slice-quaternionic Hopf manifolds of higher dimension, and that similar techniques might be possibly developed for the study of other slice-quaternionic manifolds, for example manifolds of Calabi-Eckmann type and Inoue surfaces type. This will be investigated in a forthcoming paper.
\end{rmk}

\bigskip

The plan of this note is as follows.
In Section \ref{sec:slice-reg-mfds}, we recall the notions of slice-quaternionic manifold and slice-quaternionic map.
In Section \ref{sec:slice-reg-hopf}, we study slice-quaternionic Hopf surfaces, proving Theorem \ref{thm:hopf-quat}.
In Section \ref{sec:aut}, we describe the group of automorphisms of these manifolds, Theorem \ref{thm:aut-hopf}.
Finally, in Section \ref{sec:deformations}, we construct families of slice-quaternionic Hopf surfaces.

\bigskip

\noindent{\sl Acknowledgments.}
The authors would like to warmly thank Graziano Gentili, Anna Gori, Alessandro Perotti, Giulia Sarfatti, Caterina Stoppato for interesting and useful conversations.
Thanks also to the anonymous Referee for useful comments and suggestions.

\bigskip
 
\section{Regular slice-quaternionic structures on manifolds}\label{sec:slice-reg-mfds}
Slice-quaternionic manifolds are introduced and studied in \cite{bisi-gentili, gentili-gori-sarfatti}. For sake of completeness and for reader's convenience we rephase here the definition of slice regular function of several quaternionic variables, as introduced by \cite{ghiloni-perotti} and \cite[\S2]{gentili-gori-sarfatti}, adapted to the two variables.
Let $D$ be an open subset of $\mathbb{C}^2,$ invariant with respect to complex conjugation in each variable $z_1, z_2.$ Let $\mathbb{R}_2$ be the real Clifford algebra (isomorphic to $\mathbb{H}$) with basic units $e_1, e_2$ and let $e_{12}=e_1 e_2.$
\begin{definition}
A continuous function $F \colon D \to \mathbb{H} \otimes \mathbb{R}_2,$ with $F=F_0+e_1F_1+e_2F_2+e_{12} F_{12}$ and $F_K \colon D \to \mathbb{H}$ is called a {\em stem function} if it is Clifford-intrinsic, i.e. for each $K \in \mathcal{P}(2),$ $h \in \{ 1, 2 \}$ and $z=(z_1,z_2) \in D,$ the components $F_0,F_1,F_2,F_{12}$ are, respectively, even-even, odd-even, even-odd, odd-odd with respect to $(\beta_1, \beta_2)$ where $\beta_1=\Im m (z_1)$ and $\beta_2= \Im m (z_2),$ e.g. $F_1(\overline{z_1}, z_2)=-F_1(z_1,z_2)$ and so on. 
\end{definition} 

Let $\Omega_D$ be the circular subset of $\mathbb{H}^2$ associated to $D \subset \mathbb{C}^2:$ 
$$
\Omega_D=\{ x=(x_1,x_2) \in \mathbb{H}^2 \,\, | \,\, x_h=\alpha_h + \beta_h J_h \in \mathbb{C}_{J_h}, J_h \in \mathbb{S}_{\mathbb{H}}, (\alpha_1 + i \beta_1, \alpha_2 + i \beta_2) \in D \}.
$$
\begin{definition}
Given a stem function $F \colon D \to \mathbb{H} \otimes \mathbb{R}_2,$ we define the {\em left slice function} $\mathcal{I} (F) \colon \Omega_D  \to \mathbb{H}$ induced by $F$ by setting, for each $x=(x_1,x_2)=(\alpha_1 +J_1 \beta_1, \alpha_2 +J_2 \beta_2)$
$$
\mathcal{I}(F)(x):= F_0(z_1,z_2)+J_1F_1(z_1,z_2) +J_2F_1(z_1,z_2) +J_1J_2F_{12}(z_1,z_2)  
$$
where $(z_1,z_2)=(\alpha_1 +i \beta_1, \alpha_2 + i \beta_2) \in D.$
\end{definition}
\begin{definition}\label{defsurfreg}
Let $F \colon D \to \mathbb{H} \otimes \mathbb{R}_2$ be a stem function of class $C^1$ and let $f=\mathcal{I}(F) \colon \Omega_D \to \mathbb{H}$ the induced slice function. Then $F$ is called {\em holomorphic stem function} if $\overline{\partial}_h F =0$ on $D$ for $h=1,2,$ where
$$
\overline{\partial}_1 F= \frac{1}{2} \left(\frac{\partial F}{\partial \alpha_1} + e_1 \left(\frac{\partial F}{\partial \beta_1}\right)\right)
\qquad \text{ and } \qquad
\overline{\partial}_2 F= \frac{1}{2} \left(\frac{\partial F}{\partial \alpha_2} + \left(\frac{\partial F}{\partial \beta_2}\right)e_2\right) .
$$
If $F$ is holomorphic, then we say that $f=\mathcal{I}(F)$ is a {\em (left) slice regular function} on $\Omega_D$.
\end{definition}

We consider now the geometric notion of manifold associated with slice regular functions.
\begin{defi}[{\cite[Definitions 3.1--3.2]{gentili-gori-sarfatti}}]
 Let $X$ be a differentiable manifold. Then $X$ is said to be a {\em slice(-regular)-quaternionic manifold} when endowed with a {\em slice(-regular)-quaternionic structure}, that is, a differentiable structure with changes of charts being slice regular in the sense of \cite[Definition 7]{ghiloni-perotti} (for surfaces you can also see Definition \ref{defsurfreg}).
(Note, in particular, that the domains of definition of transition functions must be circular \cite[page 736]{ghiloni-perotti}.)
 {\em Slice(-regular)-quaternionic maps} between slice-quaternionic manifolds are maps being slice regular in the sense of \cite[Definition 7]{ghiloni-perotti} in local coordinates.
\end{defi}

Note that, since the composition of slice regular functions is not slice regular in general, then the conditions have to be checked on any charts. (This motivates the possible preferred choice of special sub-classes of slice regular functions, {\itshape e.g.}, affine functions.)

By automorphism, we mean a slice-quaternionic map from $X$ to itself whose inverse is still slice-quaternionic.

In particular, slice regular functions $f\colon D \to \H$ in the sense of \cite[Definition 7]{ghiloni-perotti} in one variable on a circular domain $D\subseteq\H$ such that $D\cap\R\neq\varnothing$ are slice regular in the sense of \cite{gentili-struppa-cras, gentili-struppa}. In \cite{bisi-gentili}, slice-quaternionic structures on a $4$-real-dimensional torus are constructed and classified. More in general, since ordered polynomial functions $p(x) = \sum_{\ell} x^\ell a_\ell$ with right coefficients in $\H$ are slice regular, \cite[Examples 3.1]{ghiloni-perotti}, then any affine structure is slice-quaternionic.

\medskip

The following result allows to single out a first class of manifolds admitting quaternionic structures, see also \cite[\S3.3]{gentili-gori-sarfatti}.
For its holomorphic analogue, see, {\itshape e.g.}, \cite[2.1.5, 2.1.7]{kobayashi-horst}.
(Note indeed that we can choose circular domains as charts.)

\begin{lem}\label{lem:affine}
 Let $M$ be a differentiable manifold of dimension $4n$.
 Then $M$ admits an affine (whence slice-quaternionic) structure \cite[Definition 3.8]{gentili-gori-sarfatti} if and only if
 there is an immersion $\psi\colon \tilde M \to \H^n$ of the universal covering $\tilde M$ of $M$ such that, for every covering transformation $\gamma$ we have $\psi\circ\gamma=X_\gamma\circ\psi$ for some affine transformation $X_\gamma$ of $\H^n$.
\end{lem}

\medskip

Another class of slice-quaternionic manifolds is constructed as follows.
Let $G$ be a group of automorphisms acting on a slice-quaternionic manifold $X$.
We recall the following notions.
\begin{itemize}
 \item $G$ is called {\em properly-discontinuous} if, for any compact sets $K_1$ and $K_2$ in $X$, there are only a finite number of elements $g\in G$ such that $g(K_1)\cap K_2\neq \varnothing$.
 \item A point $x\in X$ is called a fixed-point of $g\in G$ if $g(p)=p$. The group $G$ is called {\em fixed-point free} if, for any $g\in G\setminus\{\mathrm{id}\}$, there is no fixed-point of $g$.
\end{itemize}
Analougously as in the holomorphic case, (see, {\itshape e.g.}, \cite[Theorem 2.2]{kodaira},) we have the following results.

\begin{prop}\label{prop:group-action-objs}
 Let $X$ be a slice-quaternionic manifold. Let $G$ be a group of automorphisms of $X$ such that: $G$ is properly-discontinuous; $G$ is fixed-point free.
 Then the quotient space $\left.X\middle/G\right.$ has a structure of slice-quaternionic manifold, such that the projection map $\pi\colon X \to \left.X\middle\slash G\right.$ is a slice-quaternionic map.
\end{prop}

\begin{prop}\label{prop:group-action-maps}
 Let $X_1$ and $X_2$ be slice-quaternionic manifolds. Let $G_1$ and $G_2$ be groups of automorphisms of $X_1$, respectively $X_2$, being properly-discontinuous and fixed-point free.
 Let $f\colon X_1/G_1 \to X_2/G_2$ be a slice-quaternionic map. Then there exists a slice-quaternionic map $F\colon X_1 \to X_2$ such that the diagram
 $$ \xymatrix{
  X_1 \ar[d]_{\pi_1} \ar[r]^{F} & X_2 \ar[d]^{\pi_2} \\
  \left.X_1\middle\slash G_1\right. \ar[r]_{f} & \left.X_2\middle\slash G_2\right.
 }$$
 is commutative, where $\pi_1$ and $\pi_2$ denote the natural projections.
\end{prop}

\section{Slice-quaternionic Hopf surfaces}\label{sec:slice-reg-hopf}

A {\em slice-quaternionic (primary) Hopf surface} is a slice-quaternionic manifold whose universal covering is $\H^2\setminus\{(0,0)\}$ and the fundamental group equals the infinite cyclic group $\Z$.

As a differentiable manifold, a slice-quaternionic Hopf surface $X$ is the same as the smooth manifold  underlying a complex Hopf manifold $\C^4\setminus\{0\}/\Z$, that is, it is diffeomorphic to $\mathbb S^1\times \mathbb S^7$. Its cohomology is then $H^\bullet(X;\Z)=\Z[0]\oplus\Z[-1]\oplus\Z[-7]\oplus\Z[-8]$. Note also that the differentiable manifold underlying $X$ admits a hypercomplex structure.

\medskip

We state now the main result, where we construct slice-quaternionic Hopf surfaces.
(The choice for the normal forms is assumed up to regular $*$-inverse and in accordance to the classical complex case, see \cite{sternberg}, to which we are reduced at least in the case of the generator $f$ of $\Gamma$ preserving one or all slices.)

\begin{thm}\label{thm:hopf-quat}
 Let $f\colon \H^2\to \H^2$ be the function
 $$ f(z,w) \;:=\; f_{p,\alpha,\beta,\lambda}(z,w) \;:=\; \left( z \cdot \alpha + w^p \cdot \lambda, \, w \cdot \beta \right) \;, $$
 where $p\in\N\setminus\{0\}$, $\alpha,\beta,\lambda\in\H$ are such that
 $$ 0 \;<\; \left|\alpha\right| \;\leq\; \left|\beta\right| \;<\; 1 \qquad \text{ and } \qquad \left(\alpha-\beta^p\right)\cdot \lambda \;=\; 0 \;. $$
 We consider the following cases.
 \begin{description}
  \item[Case A.] Consider the following sub-cases.
   \begin{description}
    \item[Case A.1.] When $\lambda=0$, $\alpha=\beta\in\H$ with $0<|\alpha|<1$.
    \item[Case A.2.] When $\lambda=0$, $\alpha,\beta\in\H$ with $0<|\alpha|\leq|\beta|<1$ and $\alpha\neq\beta$.
    \item[Case A.3.] When $\lambda\in\H$ with $\lambda\neq0$, $p=1$, $\alpha=\beta\in\H$ with $0<|\alpha|<1$.
   \end{description}
  \item[Case B.] When $\lambda\in\H$ with $\lambda\neq0$, $p\in\N$ with $p>1$, $\alpha=\beta^p\in\H$ with $0<|\beta|<1$.
   Assume moreover that $\beta\in\R$.
 \end{description}
 Set
 $$ \Gamma \;:=\; \left\{f^{\circ k} \;:\; k\in\Z \right\} \subseteq \mathrm{Aut}(\H^2\setminus\{(0,0)\}) \;. $$
 Then
 $$X \;:=\; X_{p,\alpha,\beta,\lambda} \;:=\; \left. \H^2 \setminus\{(0,0)\} \middle\slash \Gamma \right.$$
 is a slice-quaternionic Hopf surface.
 More precisely, it admits affine structures if and only if it belongs to case {\itshape (A)}.
\end{thm}

\begin{proof}
Case {\itshape (A)} is a consequence of Lemma \ref{lem:affine}.
Now we are reduced to prove case {\itshape (B)} by applying Proposition \ref{prop:group-action-objs}.
First, take $\beta\in\mathbb H$, possibly non-real.
By computing $f\lfloor_{\H\times L_{I_\beta}}^{\circ k}(z,w)$, we get
$$ \mathrm{ext}\left(f\lfloor_{\H\times L_{I_\beta}}^{\circ k}\right)(z,w) \;=\; \left( z \cdot \beta^{kp} + w^p\cdot\left( \sum_{\substack{\ell+m=k-1 \\ \ell,m\geq0 \text{ if } k\geq0 \\ \ell,m<0 \text{ if } k< 0}} \beta^{\ell\cdot p}\cdot\lambda\cdot\beta^{m\cdot p} \right) ,\, w\cdot \beta^k \right) \;. $$
In fact, because of the specific form of the function, this is just a regular extension of the second variable.
Note that, when $\beta\not\in\R$, then $\left\{\mathrm{ext}\left(f\lfloor_{\H\times L_{I_\beta}}^{\circ k}\right) \;:\; k\in\Z \right\}$ is not a group: indeed, since $f$ is not linear, the extension of the composition is different from the (non-slice regular) composition of extensions.
This motivates the choice for $\beta\in\R$.
For any $k\in\Z\setminus\{0\}$, the map $\mathrm{ext}\left(f\lfloor_{\H\times L_{I_\beta}}^{\circ k}\right)$ has no fixed point: otherwise, we would have $w=w\cdot \beta^k$ whence, by taking norms, $|w|=0$; therefore, we would have $z\cdot \beta^{kp}=z$ whence, by taking norms, $|z|=0$.
Moreover, $\Gamma$ is properly-discontinuous. Indeed, take $K_1$ and $K_2$ compact sets in $\H^2\setminus\{(0,0)\}$. We may assume that
$$ K_1 \;\subset\; \left\{ (z,w) \in \H^2 \;:\; r_1 < |w| < R_1 \right\} \qquad \text{ and } \qquad K_2 \;\subset\; \left\{ (z,w) \in \H^2 \;:\; r_2 < |w| < R_2 \right\} $$
where $0<r_1<R_1$ and $0<r_2<R_2$ are real numbers. Then, for any $k>\lg\frac{r_2}{r_1}/\lg|\beta|$, it holds $\mathrm{ext}\left(f\lfloor_{\H\times L_{I_\beta}}^{\circ k}\right)(K_1)\cap K_2=\varnothing$.
\end{proof}

\begin{rmk}\label{rmk:affine-cplx}
We point out that the quaternionic structure given to the Hopf surfaces in cases (A), being linear quaternionic, coincides with a complex structure on $\mathbb{C}^2$. Indeed, let $q=z+w\cdot j\in\mathbb{H}, \, a=\alpha+\beta\cdot j\in\mathbb{H}$, and $z,\,w,\,\alpha,\, \beta \in \mathbb{C}$, then:
$$q \cdot a= (z+w\cdot j)\cdot(\alpha + \beta\cdot j) =(z \cdot \alpha -w \cdot \overline{\beta}) + ( z\cdot\beta +w \cdot\overline{\alpha}) \cdot j\,, $$ 
{\itshape i.e.}, the right quaternionic multiplication by $\alpha+\beta\cdot j$ coincides with the map:
$$
\C^2 \;\ni\; (z,w) \mapsto  (z\cdot \alpha-w\cdot\overline{\beta},\, z\cdot\beta + w\cdot\overline{\alpha}) \;\in\; \C^2 \;,
$$
which in matrix form reads as:
$$
(z,w)\cdot \begin{pmatrix}
        \alpha  & \beta \\
        -\overline{\beta} &  \overline{\alpha}
        \end{pmatrix}
\;=\;
(z \cdot \alpha - w\cdot\overline{\beta}, \, z\cdot\beta +w\cdot\overline{\alpha})\;.
$$
If $\alpha\neq0$, the complex matrix $A= \begin{pmatrix}
                       \alpha & \beta \\
                       -\overline{\beta} & \overline{\alpha}
                 \end{pmatrix} $
has two different (complex conjugated) eigenvalues, hence $A$ is conjugated to a diagonal matrix.
Therefore the quaternionic actions described on the Hopf surfaces in cases (A.1) and (A.2) are particular diagonal complex actions on $\mathbb{C}^4$.
Similarly for the case (A.3).
On the other hand, the quaternionic structure on Hopf surfaces of case (B) are new, at our knowledge.
\end{rmk}

\begin{rmk}\label{rmk:other}
We wonder whether other slice-quaternionic structures on $\mathbb{S}^1\times\mathbb{S}^7$ may be constructed.

A first tentative could be by using the following extension result, that we recall.

Let $f\colon \H^2\to \H^2$ be a smooth function. Let $L_1,L_2$ be slices of $\H$ such that $f\lfloor_{L_1\times L_2}\colon \C^2 \simeq L_1\times L_2 \to \H^2$ is holomorphic. Then, any component $f^{(j)}\lfloor_{L_1\times L_2}\colon L_1\times L_2\to \mathbb H$ is analytic, let us say
$$ f^{(j)}\lfloor_{L_1\times L_2}(z_1,z_2) \;=:\; \sum_{k_1,k_2\in\N} z_1^{k_1}\cdot z_2^{k_2} \cdot a^{(j)}_{k_1,k_2} \;, $$
where $a^{(j)}_{k_1,k_2}\in \mathbb H$. We set
$$ f^* \;:=\; \mathrm{ext}(f\lfloor_{L_1\times L_2}) \;:=\; \left( \tilde f^{(1)} , \tilde f^{(2)} \right) \colon \H^2 \to \H^2 \;, $$
where
$$ \tilde f^{(j)}(z,w) \;:=\; \sum_{z,w\in\N} z^{k_1}\cdot w^{k_2} \cdot a^{(j)}_{k_1,k_2} \;. $$
Then $\mathrm{ext}(f\lfloor_{L_1\times L_2})$ is a slice regular function in the sense of \cite{ghiloni-perotti} such that
$$ \mathrm{ext}(f\lfloor_{L_1\times L_2})\lfloor_{L_1\times L_2} \;=\; f\lfloor_{L_1\times L_2} \;. $$

As we learnt from Alessandro Perotti, we have the following representation formula.

Define $f^*\colon \H^2\to \H^2$ as follows. For any
$$ x \;:=\; \left(x_1,\,x_2\right) \;=\; \left(\alpha_1+J_1\beta_1,\,\alpha_2+J_2\beta_2\right) \;\in\; \H^2 \;,$$
consider
$$ \left(y_1,\,y_2\right) \;:=\; \left(\alpha_1+I_1\beta_1,\,\alpha_2+I_2\beta_2\right) \;\in\; L_{I_1}\times L_{I_2} \;\subset\; \H^2 \;.$$
Denote, {\itshape e.g.}, by $y_1^c=(\alpha_1+I_1\beta_1)^c:=\alpha_1-I_1\beta_1$ the complex conjugate of $y_1\in L_{I_1}\simeq\C$ with respect to the complex structure $I_1$.
Define
\begin{eqnarray*}
4\cdot f^*(x) &:=& f(y_1,y_2) + f(y_1^c,y_2) + f(y_1, y_2^c) + f(y_1^c,y_2^c)  \\[5pt]
 && - J_1I_1 f(y_1,y_2) + J_1I_1 f(y_1^c,y_2) - J_1I_1 f(y_1,y_2^c) + J_1I_1 f(y_1^c,y_2^c) \\[5pt]
 && - J_2I_2 f(y_1,y_2) - J_2 I_2 f(y_1^c,y_2) + J_2I_2 f(y_1,y_2^c) + J_2I_2 f(y_1^c,y_2^c) \\[5pt]
 && + J_1J_2I_2I_1 f(y_1,y_2) - J_1J_2I_2I_1 f(y_1^c,y_2) - J_1J_2I_2I_1 f(y_1,y_2^c) + J_1J_2I_2I_1 f(y_1^c,y_2^c) \;,
\end{eqnarray*}
where $\mathbb{H}^2$ is endowed with a structure of left-$\mathbb{H}$-module by $a\cdot(b,c)=(a\cdot b,a\cdot c)$.
Then $f^*$ is the extension as before.
\end{rmk}

\begin{rmk}
 We recall that the Hopf surfaces are related to Hopf fibrations, which are important in twistor theory, see, {\itshape e.g.}, \cite{gentili-salamon-stoppato}. More precisely, consider the fibration:
 $$
 \xymatrix{
 \mathbb{S}^3 \ar@{^{(}->}[r] & \mathbb{S}^7 \ar[d]^{p} \\
 & \mathbb{H}\mathbb{P}^1
 \;,
 }
 \qquad \text{ whence } \qquad
  \xymatrix{
 \mathbb{S}^3\times\mathbb{S}^1 \ar@{^{(}->}[r] & \mathbb{S}^7 \times \mathbb{S}^1 \ar[d]^{p} \\
 & \mathbb{H}\mathbb{P}^1
 \;.
 }
 $$
 Take local charts for $U_1:=\{[z^1:z^2]\in\mathbb{HP}^1 \;:\; z^1\neq0\}\subseteq\mathbb{H}\mathbb{P}^1$ and $U_2:=\{[z^1:z^2]\in\mathbb{HP}^1 \;:\; z^2\neq0\}\subseteq\mathbb{H}\mathbb{P}^1$. The maps
 $$ \psi_1([1:\zeta^2_1],a) \;=\; \left( \frac{a}{\left(1+|\zeta_1^2|^2\right)^{1/2}},\; \frac{\zeta_1^2 \cdot a}{\left(1+|\zeta_1^2|^2\right)^{1/2}} \right) $$
 and
 $$ \psi_2([\zeta^1_2:1],a) \;=\; \left( \frac{\zeta_2^1 \cdot a}{\left(1+|\zeta_2^1|^2\right)^{1/2}} ,\, \frac{a}{\left(1+|\zeta_2^1|^2\right)^{1/2}} \right) $$
 yield diffeomorphisms
 $$
  \psi_1 \colon U_1 \times \mathbb{S}^3 \stackrel{\simeq}{\to} p^{-1}(U_1)
  \qquad \text{ and } \qquad
  \psi_2 \colon U_2 \times \mathbb{S}^3 \stackrel{\simeq}{\to} p^{-1}(U_2) \;.
 $$
\end{rmk}

\section{Automorphisms of slice-quaternionic Hopf surfaces}\label{sec:aut}

In this section, we compute the group of automorphisms of the slice-quaternionic Hopf surfaces in Theorem \ref{thm:hopf-quat}.
(For results in the holomorphic context, see \cite{namba, wehler}.)

\begin{thm}\label{thm:aut-hopf}
Let $X = \left. \H^2 \setminus\{(0,0)\} \middle\slash \langle f \rangle \right.$ be a slice-quaternionic Hopf surface. The group of automorphisms of $X$ is as follows.
\begin{description}
 \item[Case A.1.] In case $f(z,w)=\left(z\cdot \alpha, w\cdot \alpha\right)$ for $0<|\alpha|<1\colon$
 \begin{eqnarray*}
 \mathrm{Aut}(X) &=& \left\{ \varphi(z,w)=\left(z\cdot a_{1,0}+w\cdot a_{0,1},\, z\cdot b_{1,0}+w\cdot b_{0,1}\right) \right. \\[5pt]
 && \left. \left. \;:\; a_{1,0}, a_{0,1}, b_{1,0}, b_{0,1} \in L_{I_\alpha} \text{ such that } b_{0,1} \cdot a_{1,0} - b_{1,0} \cdot a_{0,1} \neq 0 \right\} \middle\slash \langle f \rangle \right. \;;
 \end{eqnarray*}
 in particular, $\dim_\R \mathrm{Aut}(X)\in\{8,16\}$.

 \item[Case A.2.1.] In case $f(z,w)=\left(z\cdot \alpha, w\cdot \beta\right)$ for $0<|\alpha|<|\beta|<1\colon$
 \begin{eqnarray*}
 \mathrm{Aut}(X) &=& \left\{ \varphi(z,w)=\left(z\cdot a_{1,0},\, w\cdot b_{0,1}\right) \right. \\[5pt]
 && \left. \left. \;:\; a_{1,0} \in L_{I_\alpha}, b_{0,1} \in L_{I_\beta} \text{ such that } b_{0,1} \cdot a_{1,0} \neq 0 \right\} \middle\slash \langle f \rangle \right. \;;
 \end{eqnarray*}
 in particular, $\dim_\R\mathrm{Aut}(X)\in\{4,6,8\}$.

\item[Case A.2.2.] In case $f(z,w)=\left(z\cdot \alpha, w\cdot \beta\right)$ for $0<|\alpha|=|\beta|<1$ and $\alpha\neq\beta\colon$
 \begin{eqnarray*}
 \mathrm{Aut}(X) &=& \left\{ \varphi(z,w)=\left(z\cdot a_{1,0} + w \cdot a_{0,1},\, z\cdot b_{1,0} + w\cdot b_{0,1}\right) \right. \\[5pt]
 && \left. \left. \;:\; a_{1,0} \in L_{I_\alpha}, b_{0,1} \in L_{I_\beta}, \beta\cdot a_{0,1}=a_{0,1}\cdot\alpha, \alpha\cdot b_{1,0}=b_{1,0}\cdot \beta \right.\right. \\[5pt]
 && \left.\left.\text{ such that } b_{1,0} \cdot ( a_{0,1} - b_{0,1} \cdot b_{1,0}^{-1} \cdot a_{1,0} ) \neq 0 \right\} \middle\slash \langle f \rangle \right. \;;
 \end{eqnarray*}
 in particular, $\dim_\R\mathrm{Aut}(X)\in\{4,6,8\}$.
 
 \item[Case A.3.] In case $f(z,w)=\left(z\cdot \alpha+w\cdot \lambda, w\cdot \alpha\right)$ for $\lambda\in\H\setminus\{0\}$ and $0<|\alpha|<1\colon$
 \begin{eqnarray*}
 \mathrm{Aut}(X) &=& \left\{ \varphi(z,w)=\left(z\cdot a_{1,0}+w\cdot a_{0,1},\, z\cdot b_{1,0}+w\cdot b_{0,1}\right) \right. \\[5pt]
 && \left. \left. \;:\; \text{ equations \eqref{eq:serie-case-a-3} hold } \right\} \middle\slash \langle f \rangle \right. \;;
 \end{eqnarray*}
 in particular, $\dim_\R\mathrm{Aut}(X)\in\{4,8\}$.

 \item[Case B.] In case $f(z,w)=\left(z\cdot \beta^p+w^p\cdot \lambda, w\cdot \beta\right)$ for $p\in\mathbb{N}\setminus\{0,1\}$, $\lambda\in\H\setminus\{0\}$ and $0<|\beta|<1$, $\beta\in\R\colon$
 \begin{eqnarray*}
 \mathrm{Aut}(X) &=& \left\{ \varphi(z,w)=\left( z\cdot b_{0,1}^p + w^p \cdot a_{0,p} ,\, w \cdot b_{0,1} \right) \right. \\[5pt]
 && \left. \left. \;:\; b_{0,1}\in\R, a_{0,p} \in \H \text{ such that } b_{0,1}\neq 0 \right\} \middle\slash \langle f \rangle \right. \;;
 \end{eqnarray*}
 in particular, $\dim_\R\mathrm{Aut}(X)=5$.
\end{description}
\end{thm}

\begin{proof}
Let $X = \left. \H^2 \setminus\{(0,0)\} \middle\slash \langle f \rangle \right.$.
Let $\varphi \colon X \to X$ be an automorphism of $X$. By Proposition \ref{prop:group-action-maps}, it is induced by an automorphism $\Phi\colon \H^2\setminus\{(0,0)\} \to \H^2\setminus\{(0,0)\}$ of the universal covering of $X$. By the Hartogs extension phenomenon for slice regular functions, \cite[Theorem 2]{ghiloni-perotti}, see also \cite[Corollary 4.9]{colombo-sabadini-struppa}, removability of singularity yields an automorphism $\Phi\colon \H^2\to\H^2$. By \cite[Corollary 2]{ghiloni-perotti}, the set of slice regular functions on $\H^2$ coincides with the one of convergent ordered power series (with right coefficients). Whence we get that
$$ \Phi(z,w) \;=\; \left( \sum_{h,k\geq0} z^h\cdot w^k\cdot a_{h,k} ,\; \sum_{h,k\geq0} z^h\cdot w^k\cdot b_{h,k} \right) \;, \qquad \text{ with } a_{0,0}\;=\;b_{0,0}\;=\;0 \;, $$
which satisfies
\begin{equation}\label{eq:comm}
\Phi \circ f \;=\; f \circ \Phi \;.
\end{equation}
We consider separately each case.

\medskip

\begin{description}
 \item[Case A.1.] {\em Consider the case $f(z,w)=\left(z\cdot \alpha, w\cdot \alpha\right)$ for $0<|\alpha|<1$.}

 \noindent By imposing \eqref{eq:comm}, we get
 \begin{eqnarray*}
 \sum_{h,k\geq0} (z\cdot \alpha)^h \cdot (w\cdot \alpha)^k \cdot a_{h,k} &=& \sum_{h,k\geq0} z^h\cdot w^k\cdot a_{h,k}\cdot\alpha \;,\\[5pt]
 \sum_{h,k\geq0} (z\cdot \alpha)^h \cdot (w\cdot \alpha)^k \cdot b_{h,k} &=& \sum_{h,k\geq0} z^h\cdot w^k\cdot b_{h,k}\cdot\alpha \;.
 \end{eqnarray*}
 
 Suppose $\alpha\not\in\R$. Since the right-hand sides are slice regular series, the left-hand sides have to be slice regular, too: that is, the non-real coefficients have to be on the right. In particular, it follows that
 $$ a_{h,k} \;=\; b_{h,k} \;=\; 0 \qquad \text{ for any } h+k\neq1 \;. $$

 Suppose now $\alpha\in\R$. Then we get
 $$
 \sum_{h,k\geq0} z^h \cdot w^k\cdot\alpha^{h+k} \cdot b_{h,k} \;=\; \sum_{h,k\geq0} z^h\cdot w^k\cdot b_{h,k}\cdot\alpha \;,
 $$
 whence, for any $h,k\in\N$,
 $$ 
 \alpha^{h+k} \cdot a_{h,k} \;=\; a_{h,k}\cdot\alpha
 \qquad \text{ and } \qquad
 \alpha^{h+k} \cdot b_{h,k} \;=\; b_{h,k}\cdot\alpha \;.
 $$
 Since $\alpha$ is real and by equalling the norms, we get that, for any $h,k\in\N$,
 $$
 |\alpha|^{h+k-1} \cdot |a_{h,k}| \;=\; |a_{h,k}|
 \qquad \text{ and } \qquad
 |\alpha|^{h+k-1} \cdot |b_{h,k}| \;=\; |b_{h,k}| \;.
 $$
 Since $|\alpha|<1$, therefore we get again
 $$ a_{h,k} \;=\; b_{h,k} \;=\; 0 \qquad \text{ for any } h+k\neq1 \;. $$

 Moreover, in both cases,
 $$
 a_{1,0}\cdot\alpha \;=\; \alpha\cdot a_{1,0}\;, \qquad
 a_{0,1}\cdot\alpha \;=\; \alpha\cdot a_{0,1}\;,
 $$
 $$
 b_{1,0}\cdot\alpha \;=\; \alpha\cdot b_{1,0}\;, \qquad
 b_{0,1}\cdot\alpha \;=\; \alpha\cdot b_{0,1}\;,
 $$
 that is,
 $$ a_{1,0},\; a_{0,1},\; b_{1,0},\; b_{0,1} \;\in\; L_{I_\alpha} \;. $$

 We recall, by \cite[Proposition 2.1]{bisi-gentili-moebius}, \cite{bisi-gentili-trends} and \cite{bisi-stoppato}, that $\Phi$ is (right-)invertible if and only if
 \begin{equation}\label{eq:invertibilita}
 b_{0,1} \cdot \left( a_{1,0} - b_{1,0} \cdot b_{0,1}^{-1} \cdot a_{0,1} \right) \;\neq\; 0
 \qquad \text{ or } \qquad
 a_{0,1} \cdot \left( b_{1,0} - a_{1,0} \cdot a_{0,1}^{-1} \cdot b_{0,1} \right) \;\neq\; 0
 \;.
 \end{equation}
 Finally we get that $\Phi$ is an automorphism if and only if
 $$
  b_{0,1} \cdot a_{1,0} - b_{1,0} \cdot a_{0,1} \;\neq\; 0 \;.
 $$
 
 Therefore $\dim_\R \mathrm{Aut}(X)\in\{8,16\}$, according to $\alpha\in\mathbb H \setminus\R$, respectively $\alpha\in\R$.

\medskip

 \item[Case A.2.] {\em Consider the case $f(z,w)=\left(z\cdot \alpha, w\cdot \beta\right)$ for $0<|\alpha|\leq|\beta|<1$ and $\alpha\neq\beta$.}
 
 \noindent By imposing \eqref{eq:comm}, we get
 \begin{eqnarray}\label{eq:serie-case-a-2}
 \sum_{h,k\geq0} (z\cdot \alpha)^h \cdot (w\cdot \beta)^k \cdot a_{h,k} &=& \sum_{h,k\geq0} z^h\cdot w^k\cdot a_{h,k}\cdot\alpha \;,\\[5pt]
 \sum_{h,k\geq0} (z\cdot \alpha)^h \cdot (w\cdot \beta)^k \cdot b_{h,k} &=& \sum_{h,k\geq0} z^h\cdot w^k\cdot b_{h,k}\cdot\beta \;.\nonumber
 \end{eqnarray}

 \smallskip

 Suppose first that both $\alpha$ and $\beta$ are non-real. Since the right-hand side is a slice regular series, the left-hand side has to be slice regular, too. 
This yields the following conditions on $a_{h,k}$ and $b_{h,k}:$
$$
a_{h,k} =0 \,\,\,\, \textrm{for any}  \,\,\,\, (h,k) \neq \{ (1,0), (0,1) \},
$$ 
$$
b_{h,k} =0 \,\,\,\, \textrm{for any}  \,\,\,\, (h,k) \neq \{ (1,0), (0,1) \}.
$$
In the other cases, when either $\alpha$ or $\beta$ or both are real, with similar arguments as before, we recover the same conditions.

Now, we distinguish two cases.
 \item[Case A.2.1.] 
 In case $|\alpha|\neq|\beta|$, by equalling the norms, we get that
 $$ a_{0,1} \;=\; 0, $$
 $$ b_{1,0} \;=\; 0. $$
 Moreover,
 $$
 a_{1,0}\cdot\alpha \;=\; \alpha\cdot a_{1,0}\;, \quad
 b_{0,1}\cdot\beta \;=\; \beta\cdot b_{0,1}\;.
 $$
 that is,
 $$ a_{1,0} \;\in\; L_{I_\alpha}\;,\qquad b_{0,1} \;\in\; L_{I_\beta} \;. $$
Finally, we get that $\Phi$ is an automorphism if and only if
 $$
  b_{0,1} \cdot a_{1,0} \;\neq\; 0 \;.
 $$
In the general case that both $\alpha$ and $\beta$ are non-real, then $\dim_{\mathbb{R}} \mathrm{Aut}(X) =4.$ If only one among $\alpha$ and $\beta$ is real, then $\dim_{\mathbb{R}} \mathrm{Aut}(X)=6$ and if both $\alpha$ and $\beta$, are real then $\dim_{\mathbb{R}} \mathrm{Aut}(X)=8$.
\item[Case A.2.2.]
In case $|\alpha|=|\beta|$, by equalling the norms, we always get that
 $$ a_{h,k} \;=\; 0 \qquad \text{ for any } h+k\neq1 \;, $$
 $$ b_{h,k} \;=\; 0 \qquad \text{ for any } h+k\neq1 \;. $$
 Moreover,
 $$
 a_{1,0}\cdot\alpha \;=\; \alpha\cdot a_{1,0}\;, \quad
 b_{0,1}\cdot\beta \;=\; \beta\cdot b_{0,1}\;,
 $$
 that is,
 $$ a_{1,0} \;\in\; L_{I_\alpha}\;,\qquad b_{0,1} \;\in\; L_{I_\beta} \;; $$
 but also
 $$
 a_{0,1}\cdot\alpha \;=\; \beta\cdot a_{0,1}\;, \quad
 b_{1,0}\cdot\beta \;=\; \alpha\cdot b_{1,0}\;.
 $$
 Note that the map, {\itshape e.g.}, $a_{0,1} \mapsto \beta\cdot a_{0,1}\cdot\alpha^{-1}$ is given by the composition of two rotations along two orthogonal planes, with angles given by the sum, respectively the difference, of the arguments of $\alpha$ and $\beta$. Then its fixed-locus is either a point or a plane: it is a plane iff $\beta=\overline{\alpha},$ and it is a point in all remaining cases.
 
 Hence if $\alpha, \beta \in \mathbb{H} \setminus \mathbb{R}$ and $\alpha = \overline{\beta}$ then $\dim_{\mathbb{R}} \mathrm{Aut}(X) = 8$; if $\alpha, \beta \in \mathbb{H} \setminus \mathbb{R}$ and $\alpha \neq \overline{\beta}$ then $\dim_{\mathbb{R}} \mathrm{Aut}(X) = 4$;  if $\alpha,$ or $\beta,$ or both, are real, $a_{0,1}=b_{1,0}=0$ and $\dim_{\mathbb{R}} \mathrm{Aut}(X)$ is equal to $6$, respectively $8$.
 
 Finally, we get that $\Phi$ is an automorphism if and only if
 $$
  b_{1,0} \cdot ( a_{0,1} - b_{0,1} \cdot b_{1,0}^{-1} \cdot a_{1,0} ) \;\neq\; 0 \;.
 $$

 \medskip
 
 \item[Case A.3.] {\em Consider the case $f(z,w)=\left(z\cdot \alpha+w \cdot \lambda, w\cdot \alpha\right)$ for $\lambda\in\H\setminus\{0\}$ and $0<|\alpha|<1$.}
  
 \noindent By imposing \eqref{eq:comm}, we get
 \begin{eqnarray*}
 \sum_{h,k\geq0} (z\cdot \alpha + w \cdot \lambda)^h \cdot (w\cdot \alpha)^k \cdot a_{h,k} &=& \sum_{h,k\geq0} z^h\cdot w^k\cdot (a_{h,k}\cdot\alpha + b_{h,k}\cdot\lambda) \;,\\[5pt]
 \sum_{h,k\geq0} (z\cdot \alpha + w \cdot \lambda)^h \cdot (w\cdot \alpha)^k \cdot b_{h,k} &=& \sum_{h,k\geq0} z^h\cdot w^k\cdot b_{h,k}\cdot\alpha \;.\nonumber
 \end{eqnarray*}
 
 We restrict it to the quaternionic line $r_\mu := \{(z,z \cdot \mu) \;:\; z\in\mathbb H\}$ where $\mu\in\R$ is any fixed real number:
 \begin{eqnarray*}
 \sum_{h,k\geq0} (z\cdot(\alpha + \mu \cdot \lambda))^h \cdot (z\cdot \mu \cdot \alpha)^k \cdot a_{h,k} &=& \sum_{h,k\geq0} z^h\cdot (z\cdot \mu)^k\cdot (a_{h,k}\cdot\alpha + b_{h,k}\cdot\lambda) \;,\\[5pt]
 \sum_{h,k\geq0} (z\cdot (\alpha + \mu \cdot \lambda))^h \cdot (z\cdot \mu \cdot \alpha)^k \cdot b_{h,k} &=& \sum_{h,k\geq0} z^h\cdot (z\cdot \mu)^k\cdot b_{h,k}\cdot\alpha \;.\nonumber
 \end{eqnarray*}
 
 \smallskip

 Suppose first that $\alpha$ is non-real. For $\mu$ in a dense subset of $\R$, it holds that also $\alpha+\mu\cdot\lambda\not\in\R$.
 We notice that the right-hand sides are slice regular. Notwithstanding, the left-hand sides are slice regular if and only if
 $$ a_{h,k} \;=\; 0 \qquad \text{ for any } h+k\neq1 \;, $$
 $$ b_{h,k} \;=\; 0 \qquad \text{ for any } h+k\neq1 \;. $$
 
 So we are reduced to:
 \begin{eqnarray*}
 (\alpha + \mu \cdot \lambda) \cdot a_{1,0}
 + \mu \cdot \alpha \cdot a_{0,1}
 &=&
 (a_{1,0}\cdot\alpha + b_{1,0}\cdot\lambda)
 + \mu \cdot (a_{0,1}\cdot\alpha + b_{0,1}\cdot\lambda)
 \;,\\[5pt]
 (\alpha + \mu \cdot \lambda) \cdot b_{1,0}
 + \mu \cdot \alpha \cdot b_{0,1}
 &=& b_{1,0}\cdot\alpha
 + \mu \cdot b_{0,1}\cdot\alpha \;.
 \end{eqnarray*}
 By equalling the coefficients in the polynomial in $\mu$, we get:
 \begin{eqnarray}\label{eq:serie-case-a-3}
 \lambda \cdot a_{1,0}
 + \alpha \cdot a_{0,1}
 &=&
 a_{0,1}\cdot\alpha
 + b_{0,1}\cdot\lambda
 \;,\\[5pt]
 \alpha \cdot a_{1,0} &=&
 a_{1,0} \cdot\alpha
 + b_{1,0}\cdot\lambda \;, \nonumber \\[5pt]
 \lambda \cdot b_{1,0}
 + \alpha \cdot b_{0,1}
 &=& b_{0,1}\cdot\alpha \;, \nonumber \\[5pt]
 \alpha \cdot b_{1,0}
 &=& b_{1,0}\cdot\alpha \;. \nonumber 
 \end{eqnarray}
 The last condition gives:
 \begin{itemize}
  \item $b_{1,0}\in L_{I_{\alpha}}\simeq\C$.
 \end{itemize}
 The other conditions give:
 \begin{itemize}
  \item $b_{0,1}$ as fixed point of $X \mapsto \alpha \cdot X \cdot \alpha^{-1} + \lambda \cdot b_{1,0} \cdot \alpha^{-1}$,
  \item $a_{1,0}$ as fixed point of $X \mapsto \alpha \cdot X \cdot \alpha^{-1} - b_{1,0} \cdot \lambda \cdot \alpha^{-1}$,
  \item $a_{0,1}$ as fixed point of $X \mapsto \alpha \cdot X \cdot \alpha^{-1} + \lambda \cdot a_{1,0} \cdot \alpha^{-1} - b_{0,1} \cdot \lambda \cdot \alpha^{-1}$.
 \end{itemize}
 We add further the condition of invertibility \eqref{eq:invertibilita}.
 
The dimension of the group of automorphisms in this case can be computed reasoning in the following way.
\begin{itemize}
\item If $\lambda \cdot b_{1,0} \cdot \alpha^{-1} \in L_{I_\alpha}$, equivalently, $\lambda\in L_{I_\alpha}$ in case $b_{1,0}\neq0$, (for example, if $\lambda \in \mathbb{R}$,) then $b_{0,1}$ cannot be in $L_{I_\alpha}^{\perp}$ because it is a 2-plane invariant by the rotation $X \to \alpha \cdot X \cdot \alpha^{-1}$ therefore no point on it can be fixed by such a rotation followed by a translation of an element in $L_{I_{\alpha}}.$ For the same reason, the component of $b_{0,1}$ along $L_{I_{\alpha}}^{\perp}$ has to be zero. On the other hand, $b_{0,1}$ can be in $L_{I_{\alpha}}$ if and only if $b_{1,0}=0$ because the rotation $X \to \alpha \cdot X \cdot \alpha^{-1}$ fixes point by point the elements of $L_{I_\alpha}$; in this case $b_{0,1}$ can be chosen arbitrarily in $L_{I_{\alpha}}.$ 
In this case also $a_{1,0}$ can be chosen arbitrarily in $L_{I_{\alpha}}$ because $b_{1,0}$ is zero also in the second map, and from the third map it follows that $a_{1,0}=b_{0,1}$ and $a_{0,1}$ is arbitrarily chosen in $L_{I_{\alpha}}.$ Hence $\dim_{\mathbb{R}} \mathrm{Aut} (X) = 4$.
\item If $\lambda$ is such that $\lambda \cdot b_{1,0} \cdot \alpha^{-1} \in L_{I_{\alpha}}^{\perp}$, (equivalently, $\lambda \in L_{I_{\alpha}}^{\perp}$ in case $b_{1,0}\neq0$,) then the map $X \mapsto \alpha \cdot X \cdot \alpha^{-1} +\lambda \cdot b_{1,0} \cdot \alpha^{-1}$ is a roto-translation of the invariant $2$-plane $L_{I_{\alpha}}^{\perp}$ which has a unique fixed point on that $2$-plane and in this case the component of $b_{0,1}$ along $L_{I_\alpha}^\perp$ has to be this fixed point which depends on $b_{1,0}$. Notice that also $b_{1,0}\cdot\lambda\cdot\alpha^{-1}$ is in $L_{I_{\alpha}}^\perp$, so, for the same reason, also the component of $a_{1,0}$ along $L_{I_{\alpha}}^{\perp}$ has to be the unique fixed point of a roto-translation of the invariant $2$-plane $L_{I_{\alpha}}^{\perp}$  depending on $b_{1,0}$. Finally, the last equation yields the condition that the component along $L_{I_\alpha}$ of $\lambda\cdot a_{1,0}-b_{0,1}\cdot\lambda$ is zero. 
Since $\lambda \in L_{I_{\alpha}}^\perp$, this last condition is translated in $\lambda\cdot a_{1,0}^{\perp, \alpha} -b_{0,1}^{\perp, \alpha} \cdot\lambda=0$ where $a_{1,0}^{\perp, \alpha}$ $b_{0,1}^{\perp, \alpha}$ are respectively the components of $a_{1,0}$ and $b_{0,1}$ along $L_{I_{\alpha}}^{\perp}$.
If the equation is not satisfied, then necessarily $b_{1,0}=0,$ then we are again in the previous case. Otherwise, we have $\dim_{\mathbb{R}} \mathrm{Aut} (X) = 2+2+2+2 = 8$.
\item Finally, in the general case, split $X=X_1+X_2$ with $X_1\in L_{I_\alpha}$ and $X_2\in L_{I_\alpha}^\perp$.
If the component of $\lambda \cdot b_{1,0} \cdot \alpha^{-1}$ along $L_{I_\alpha}$ is zero, then $\lambda \cdot b_{1,0} \cdot \alpha^{-1} \in L_{I_\alpha}^\perp$, so we are back to the previous case.
If the component of $\lambda \cdot b_{1,0} \cdot \alpha^{-1}$ along $L_{I_\alpha}$ is non-zero, then this implies, as before, that $\lambda \cdot b_{1,0} \cdot \alpha^{-1}=0$. Note that this happens exactly when the component of $\lambda$ along $L_{I_\alpha}^\perp$ is non-zero. Then we argue as in the previous case, getting $\dim_{\mathbb{R}} \mathrm{Aut} (X) \in \{4,8\}$.
\end{itemize} 
 
 \smallskip
 
 Suppose now $\alpha\in\R$, and that $\lambda\not\in\R$. By imposing the slice regularity of the left-hand side, we get that: $a_{h,k}=0$ and $b_{h,k}=0$ for $(h,k)\not\in\{(1,0),\, (0,k) \;:\; k\geq1\}$.
 Therefore we are reduced to:
 \begin{eqnarray*}
 \lefteqn{ z \cdot (\alpha+\mu\cdot\lambda)\cdot a_{1,0}+\sum_{k\geq1}z^k\cdot\mu^k\cdot\alpha^k\cdot a_{0,k} } \\[5pt]
 &=& z \cdot (a_{1,0}\cdot\alpha+b_{1,0}\cdot\lambda)+\sum_{k\geq1} z^{k}\cdot \mu^k\cdot (a_{0,k}\cdot\alpha + b_{0,k}\cdot\lambda) \;,\\[5pt]
 \lefteqn{ z\cdot (\alpha + \mu \cdot \lambda)\cdot b_{1,0}+\sum_{k\geq1} z^k\cdot \mu^k \cdot \alpha^k \cdot b_{0,k} } \\[5pt]
 &=& z \cdot b_{1,0}\cdot\alpha+\sum_{k\geq1} z^{k}\cdot \mu^k\cdot b_{0,k}\cdot\alpha \;.\nonumber
 \end{eqnarray*}
 We equal: firstly, the coefficients in $z$; secondly, the coefficients in $\mu$. We are reduced to:
 \begin{eqnarray*}
  \alpha\cdot a_{1,0} &=& a_{1,0}\cdot\alpha+b_{1,0}\cdot\lambda \;,\\[5pt]
  \lambda\cdot a_{1,0} + \alpha\cdot a_{0,1} &=& a_{0,1}\cdot\alpha+ b_{0,1}\cdot \lambda \;,\\[5pt]
  \alpha^k \cdot a_{0,k} &=& a_{0,k}\cdot \alpha + b_{0,k}\cdot \lambda \;, \qquad \text{ for } k>1\;,\\[5pt]
  \alpha \cdot b_{1,0} &=& b_{1,0}\cdot\alpha \;,\\[5pt]
  \lambda\cdot b_{1,0}+ \alpha \cdot b_{0,1} &=& b_{0,1}\cdot\alpha \;,\\[5pt]
  \alpha^k \cdot b_{0,k} &=& b_{0,k}\cdot\alpha \;, \qquad \text{ for } k>1 \;.
 \end{eqnarray*}
 From the first equation, $\alpha$ being real, we get that
 $$ b_{1,0} \;=\; 0 \;. $$
 From the last equation, since $|\alpha|<1$, we get that $b_{0,k}=0$ for $k>1.$
 Whence, from the third equation, the same argument gives $a_{0,k}=0$ for $k>1$.
 Therefore we are reduced to:
 $$ a_{1,0}\in\mathbb{H} \;, \qquad a_{0,1}\in\mathbb{H}\;, \qquad b_{0,1} \;=\; \lambda\cdot a_{1,0}\cdot\lambda^{-1} \;,$$
 the other coefficients are zero,
 with the condition of invertibility \eqref{eq:invertibilita}. Hence $\dim_{\mathbb{R}} \,\mathrm{Aut}(X)=8.$
 
 \smallskip
 
 The last case is when $\alpha\in\R$ and $\lambda\in\R$. We get the equations:
 \begin{eqnarray*}
 \sum_{h,k\geq0} z^{h+k} \cdot (\alpha+\mu\cdot\lambda)^h\cdot \mu^k \cdot \alpha^k\cdot a_{h,k} &=& \sum_{h,k\geq0} z^{h+k} \cdot \mu^k \cdot (a_{h,k}\cdot\alpha+b_{h,k}\cdot\lambda) \;,\\[5pt]
 \sum_{h,k\geq0} z^{h+k} \cdot (\alpha+\mu\cdot\lambda)^h\cdot \mu^k \cdot \alpha^k\cdot b_{h,k} &=& \sum_{h,k\geq0} z^{h+k} \cdot \mu^k\cdot b_{h,k}\cdot\alpha \;.\nonumber
 \end{eqnarray*}

 In the second equation, we compare the coefficients in $z^t$ for $t=h+k>1$ and then in $\mu^0$, getting $\alpha^t\cdot b_{t,0}=b_{t,0}\cdot\alpha$, whence:
 $$ b_{h,0}\;=\;0 \qquad \text{ for } h>1 \;. $$
 Then, in the second equation, we compare the coefficients in $z^t$ for $t=h+k>1$ and then in $\mu$, getting $\alpha^{t-1}\cdot\alpha\cdot b_{t-1,1}+t\cdot\alpha\cdot\lambda\cdot b_{t,0}= b_{t-1,1}\cdot\alpha$, whence:
 $$ b_{h-1,1} \;=\; 0 \qquad \text{ for } h>1 \;. $$
 By induction on $\ell\geq0$, by comparing the coefficients in $z^t$ for $t>1$ and then in $\mu^\ell$, we get that:
 $$ b_{h-\ell,\ell} \;=\; 0 \qquad \text{ for } h>1,\; \ell\geq 0 \;. $$
 Now, in the first equation, we compare the coefficients in $z^t$ for $t>1$ and then in $\mu^0$, we get that $a_{h,0}=0$ for $h>1$; proceeding by induction as before, we finally get:
 $$ a_{h-\ell,\ell}\;=\; 0 \qquad \text{ for }h>1,\; \ell\geq0 \;. $$
 Finally, by looking at the degree $h+k=1$, we have that:
 $$ a_{1,0} \;\in\; \mathbb{H} \;, \qquad a_{0,1} \;\in\; \mathbb{H}\;, \qquad b_{1,0} \;=\; 0 \;, \qquad b_{0,1} \;=\; a_{1,0} \;, $$
 the other coefficients are zero, and we assume the invertibility condition \eqref{eq:invertibilita}. Hence $\dim_{\mathbb{R}} \mathrm{Aut}(X) =8$.

\medskip

 \item[Case B.] {\em Consider the case $f(z,w)=\left(z\cdot \beta^p+w^p \cdot \lambda, w\cdot \beta\right)$ for $p\in\mathbb{N}\setminus\{0,1\}$, $\lambda\in\H\setminus\{0\}$ and $0<|\beta|<1$, $\beta\in\R$.}
 
 \medskip
 
 \noindent Consider first the case $\lambda\in\R$. By imposing \eqref{eq:comm} on the quaternionic curve $\gamma_\mu:=\{ (z^p, \mu\cdot z) \;:\; z\in\mathbb{H}\}$ for a fixed $\mu\in\R$, we get:
 \begin{eqnarray*}
  \lefteqn{ \sum_{h,k} \left( z^{ph+k} \cdot \mu^k \cdot a_{h,k} \right) \cdot \beta^p + \left( \sum_{h,k} z^{ph+k} \cdot \mu^k \cdot b_{h,k} \right)^p \cdot \lambda } \\[5pt]
  &=&
  \sum_{h,k} z^{ph+k} \cdot (\beta^p+\mu^p\cdot\lambda)^h \cdot \mu^k\cdot \beta^k\cdot a_{h,k} \;, \\[5pt]
  \lefteqn{ \sum_{h,k} z^{ph+k} \cdot \mu^k \cdot b_{h,k} \cdot \beta } \\[5pt]
  &=& \sum_{h,k} z^{ph+k} \cdot (\beta^p+\mu^p\cdot\lambda)^h\cdot\mu^k\cdot\beta^k\cdot b_{h,k} \;.
 \end{eqnarray*}
 We argue now as in case {\itshape (A.3)}, the only difference being that we consider $h$ with weight $p$.
 Namely, from the second equation, we compare the coefficients in $z^t$ for $t=p\cdot h+k>1$ and we get that
 $$ b_{h,k}\;=\;0 \qquad \text{ for } \qquad (h,k) \not\in \{ (1,0), \, (0,1), \, \ldots, \, (0,p) \} \;.$$
 Notice also that, in order to the left-hand side of the first equation being slice-regular, we need
 $$ b_{1,0} \;\in\; \R \;, \qquad b_{0,1} \;\in\; \R \;, \qquad \ldots \;, \qquad b_{0,p} \;\in\; \R \;. $$
 The same argument again, applied to the first equation, gives:
 $$ a_{h,k}\;=\;0 \qquad \text{ for } \qquad (h,k) \not\in \{ (1,0), \, (0,1), \, \ldots,\, (0,p) \} \;.$$
 Moreover, we get the conditions:
 \begin{eqnarray*}
  a_{0,k} &=& 0 \quad \text{ for } k\in\{1,\ldots,p-1\} \;, \\[5pt]
  b_{0,1}^{p} &=& a_{1,0} \;, \\[5pt]
  b_{0,k} &=& 0 \quad \text{ for } k\in\{2,\ldots,p-1\} \;, \\[5pt]
  b_{1,0}^p &=& 0 \;, \\[5pt]
  b_{0,p} &=& 0 \;.
 \end{eqnarray*}
 Finally, we have
 $$ \Phi (z,w) \;=\; \left( z\cdot b_{0,1}^p + w^p \cdot a_{0,p} ,\; w \cdot b_{0,1} \right) $$
 where
 $$ b_{0,1} \;\in\; \R \qquad \text{ and } \qquad a_{0,p} \;\in\; \mathbb{H} \; . $$
The invertibility of $\Phi$ is guaranteed if $b_{0,1} \in \mathbb{R}\setminus\{0\}$.
 
 Consider now the case that $\lambda\not\in\R$. As before, we impose \eqref{eq:comm} on the quaternionic curve $\gamma_\mu:=\{(z^p,\mu\cdot z) \;:\; z\in\mathbb{H}\}$, and look at the second component. We have:
 $$ \sum_{h,k} z^{ph+k} \cdot \mu^k \cdot b_{h,k} \cdot \beta \;=\; \sum_{h,k} \left(z^p \cdot (\beta^p + \mu^p \cdot \lambda) \right)^h \cdot z^k \cdot \mu^k \cdot \beta^k \cdot b_{h,k} \;. $$
 The slice regularity of the left-hand side forces
 $$ b_{h,k} \;=\; 0 \qquad \text{ for } \qquad (h,k)\not\in\{(0,k),(1,0)\} \;. $$
 Arguing as we did before in the case $\lambda\in\R$, we get:
 $$ b_{h,k} \;=\; 0 \qquad \text{ for } \qquad (h,k)\neq(0,1) \;. $$
 Look now at the first component of \eqref{eq:comm} on the quaternionic curve $\gamma_\mu$:
 $$ \sum_{h,k} z^{ph+k} \cdot \mu^k \cdot a_{h,k} \cdot \beta^p + (z \cdot \mu \cdot b_{0,1})^p \cdot \lambda \;=\; \sum_{h,k} \left( z^p \cdot (\beta^p+\mu^p\cdot\lambda)\right)^h \cdot z^k \cdot \beta^k \cdot a_{h,k} \;.$$
 The slice regularity of the first series forces
 $$ b_{0,1} \;\in\; \R \; \qquad \text{ and } \qquad a_{h,k} \;=\; 0 \quad \text{ for } \quad (h,k)\not\in\{(0,k),(1,0)\} \;. $$
 In order to have an automorphism, $b_{0,1} \in \mathbb{R}\setminus \{ 0 \}$.
 We argue as before to conclude that the conditions for automorphisms are the same.
\end{description}
This concludes the proof.
\end{proof}

\section{Families of slice-quaternionic structures and deformations}\label{sec:deformations}
In this section, we construct families of slice-quaternionic Hopf surfaces, connecting cases {\itshape (A.1)} and {\itshape (A.3)}, respectively cases {\itshape (A.2.1)} and {\itshape (B)}, in the notation of Theorem \ref{thm:hopf-quat}.
These are examples of deformations in the slice-quaternionic setting.
(For the holomorphic analogue, see \cite[Example 2.15]{kodaira}.)

\medskip

We start constructing a {\em smooth} family of slice-quaternionic Hopf surfaces, connecting cases {\itshape (A.1)} and {\itshape (A.3)}, in Theorem \ref{thm:hopf-quat}.

Consider
$$ \mathbb H^2\setminus\{(0,0)\}\times\mathbb H \;. $$
Fix $\alpha\in\mathbb H$ such that $0<|\alpha|<1$. Consider the slice-regular function
$$ F\colon \mathbb H^2\setminus\{(0,0)\}\times\mathbb H \to \mathbb H^2\setminus\{(0,0)\}\times\mathbb H \;,$$
$$ F(z,w,\lambda) \;:=\; F_{1,\alpha,\alpha}(z,w,\lambda) \;:=\; (z\cdot \alpha+w\cdot \lambda,\, w\cdot \alpha, \, \lambda) \;, $$
where $\alpha\in\mathbb H$ is such that $0<|\alpha|<1$.
Define
$$
\mathcal{M} \;:=\; \left. \mathbb H^2\setminus\{(0,0)\}\times\mathbb H \middle\slash \Gamma \right.
\qquad
\text{ where }
\qquad
\Gamma \;:=\; \left\{ F^{\circ k} \;:\; k\in\Z \right\} \;.
$$
Since the action of $\Gamma$ is fixed-point free and properly discontinuous, then $\mathcal{M}$ is a smooth manifold.
(Note that, for $\alpha\in\R$, then any $F^{\circ k}$ is slice-regular, then we get a slice-quaternionic family.)

The slice-quaternionic projection
$$ \pi \colon \mathbb H^2\setminus\{(0,0)\}\times\mathbb H \to \mathbb H $$
makes the following diagram commutative:
$$\xymatrix{
\mathbb H^2\setminus\{(0,0)\}\times\mathbb H \ar[dr]_{\pi} \ar[rr]^{F_{1,\alpha,\alpha}} && \mathbb H^2\setminus\{(0,0)\}\times\mathbb H \ar[dl]^{\pi} \\
&\mathbb H &
}$$
whence it induces a map
$$ \pi \colon \mathcal{M} \to \mathbb H \;, $$
whose fibres are the slice-regular manifolds
$$ \pi^{-1}(\lambda) \;=\; X_{1,\alpha,\alpha,\lambda} \;. $$

Note how the group of automorphisms changes:
$$ \dim_\R \mathrm{Aut}(\pi^{-1}(0))\;\in\;\{8,16\}\;\qquad\text{ and }\qquad\dim_\R\mathrm{Aut}(\pi^{-1}(\lambda))\;\in\;\{4,8\}\;.$$
In particular, assume $\alpha\in\R$ as a special example:
$$ \dim_\R \mathrm{Aut}(\pi^{-1}(0))\;=\;16 \;\qquad\text{ and }\qquad \dim_\R\mathrm{Aut}(\pi^{-1}(\lambda))\;=\;8 \;.$$

\medskip

We construct now a {\em slice-quaternionic} family of slice-quaternionic Hopf surfaces, connecting cases {\itshape (A.2.1)} and {\itshape (B)}, respectively in Theorem \ref{thm:hopf-quat}.

Consider
$$ \mathbb H^2\setminus\{(0,0)\}\times\mathbb H \;. $$
Fix $p\in\N$ with $p>1$, and $\beta\in\R$ such that $0<|\beta|<1$, and take $\alpha=\beta^p$. Consider the slice-regular function
$$ F\colon \mathbb H^2\setminus\{(0,0)\}\times\mathbb H \to \mathbb H^2\setminus\{(0,0)\}\times\mathbb H \;,$$
$$ F(z,w,\lambda) \;:=\; (z\cdot \beta^p+w^p\cdot \lambda,\, w\cdot \beta, \, \lambda) \;. $$
Define
$$
\mathcal{M} \;:=\; \left. \mathbb H^2\setminus\{(0,0)\}\times\mathbb H \middle\slash \Gamma \right.
\qquad
\text{ where }
\qquad
\Gamma \;:=\; \left\{ F^{\circ k} \;:\; k\in\Z \right\} \;.
$$
Since the action of $\Gamma$ is fixed-point free and properly discontinuous, then $\mathcal{M}$ is a slice-quaternionic manifold by Proposition \ref{prop:group-action-objs}.

The slice-quaternionic projection
$$ \pi \colon \mathbb H^2\setminus\{(0,0)\}\times\mathbb H \to \mathbb H $$
makes the following diagram commutative:
$$\xymatrix{
\mathbb H^2\setminus\{(0,0)\}\times\mathbb H \ar[dr]_{\pi} \ar[rr]^{F} && \mathbb H^2\setminus\{(0,0)\}\times\mathbb H \ar[dl]^{\pi} \\
&\mathbb H &
}$$
whence it induces a slice-quaternionic map
$$ \pi \colon \mathcal{M} \to \mathbb H \;, $$
whose fibres are
$$ \pi^{-1}(\lambda) \;=\; X_{p,\beta^p,\beta,\lambda} \;. $$
They are in case {\itshape (B)} for $\lambda\neq0$. In particular, for $\lambda=0$, we have:
$$ \pi^{-1}(0) \;=\; X_{p,\beta^p,\beta,0} \;=\; X_{1,\beta^p,\beta,0} \;, $$
which is in case {\itshape (A.2.1)} with $\beta \in \mathbb{R},$ $\alpha=\beta^p \in \mathbb{R},\,\,\, \alpha \neq \beta,$ and which is different from $\pi^{-1}(\lambda)$, for $\lambda\neq0$, because of Theorem \ref{thm:aut-hopf}.

Note how the group of automorphisms changes:
 \begin{eqnarray*}
 \mathrm{Aut}(\pi^{-1}(0)) &=& \left\{ \varphi(z,w)=\left(z\cdot a_{1,0},\, w\cdot b_{0,1}\right) \right. \\[5pt]
 && \left. \left. \;:\; a_{1,0} , b_{0,1} \in \mathbb{H} \text{ such that } b_{0,1} \cdot a_{1,0} \neq 0 \right\} \middle\slash \langle f \rangle \right. \;,
 \end{eqnarray*}
and, for $\lambda\neq0$,
 \begin{eqnarray*}
 \mathrm{Aut}(\pi^{-1}(\lambda)) &=& \left\{ \varphi(z,w)=\left(z\cdot b_{0,1}^p+w^p\cdot a_{0,p},\, w\cdot b_{0,1}\right) \right. \\[5pt]
 && \left. \left. \;:\; b_{0,1}\in\R, a_{0,p} \in \H \text{ such that } b_{0,1}\neq 0 \right\} \middle\slash \langle f \rangle \right. \;.
 \end{eqnarray*}
In particular,
$$ \dim_\R\mathrm{Aut}(\pi^{-1}(0))\;=\;8 \qquad \text{ and } \qquad \dim_\R\mathrm{Aut}(\pi^{-1}(\lambda))\;=\;5\;.$$

\end{document}